\documentclass[12pt]{article}

\usepackage{latexsym}
\usepackage{euscript}
\usepackage{amssymb}
\usepackage{amsthm}
\usepackage{enumerate}
\usepackage{amsmath}
\usepackage{authblk}

\newcommand{\SH}{\mathcal{SH}}
\newcommand{\LL}{\mathbf{L}}

\newtheorem{definition}{\bf Definition}[section]

\newtheorem{lemma}{\bf Lemma}[section]

\newtheorem{corollary}{\bf Corollary}[section]

\title{A recursive formula for the number of semi-Heyting algebras definable on a finite chain}
\author{Luiz F.~Monteiro}
\affil{Universidad Nacional del Sur}
\author{Juan M.~Cornejo}
\author{Ignacio D.~Viglizzo} 
\affil{INMABB-UNS-CONICET, Departamento de Matemática, Av. Alem 1253, Bahía Blanca, Buenos Aires, Argentina.}
\affil[ ]{\tt lfmonteiro0510@gmail.com, jmcornejo@uns.edu.ar, viglizzo@gmail.com}

\date{}
\begin{document}
\maketitle

\begin{flushright}         
\textit{To our friend Hanamantagouda P.~Sankappanavar}
\end{flushright}

\thispagestyle{empty}
\font\fivrm=cmr5 \relax
\input{prepictex}
\input{pictex}
\input{postpictex}

\begin{abstract}
	We provide a recursive construction of all the semi-Heyting algebras that can be defined on a chain with $n$ elements. This construction allows us to count them easily. We also compare the formula for the number of semi-Heyting chains thus obtained to the one previously known. 
\end{abstract}

\section{Preliminaries}

\begin{definition}
{\rm \cite{sankappanavar1985semi}} An algebra $\LL = \langle L, \vee, \wedge, \rightarrow, 0, 1\rangle$ is a semi-Heyting algebra if the following conditions hold:
\begin{quote}
	\begin{enumerate}
		\item[$(SH 1)$] $\langle L, \vee, \wedge, 0, 1\rangle $ is a lattice with 0 and 1.
		\item[$(SH 2)$] $x \wedge (x \rightarrow y) \approx x \wedge y$.
		\item[$(SH 3)$] $x \wedge (y \rightarrow z) \approx x \wedge [(x \wedge y) \rightarrow (x \wedge
		z)]$.
		\item[$(SH 4)$] $x \rightarrow x \approx 1$.
	\end{enumerate}
\end{quote}
\end{definition}
\medskip

We will denote by $\mathbf{\SH}$ the variety of semi-Heyting algebras.

 Let $C_n=\{0= a_0, a_1,\cdots, a_{n-2}, a_{n-1}=1\}$ be the chain with $n$ elements, with $a_0<a_1<\ldots<a_{n-1}$. We denote with $|X|$ the cardinality of a set $X$, and if $X$ is a poset, and $y\in X$, then we write $[y)$ to denote the set $\{x\in X:y\le x\}$.   Therefore,
\begin{equation}
\label{factI}
 |[a_i)|= n-i,\; 0\leq i\leq n-1.
 \end{equation}
 
 We denote with $N(n)$ the number of different implication operations that may be defined on the chain $C_n$ so that it becomes a semi-Heyting algebra.
 
If $(C_n,\to)$ is in $\SH$, then we associate to it a matrix of size $n\times n$, $M=(m_{(i,j)})$ where $m_{(i,j)}= a_i \to a_j,$ for $0\leq i,j\leq n-1$, this is the table of the implication operation. Then $m_{(i,i)}=a_i\to a_i{\stackrel{{\rm (SH4)}}{=}}1$.	

We know from  \cite{sankappanavar2008semi} Lemma 4.2, (iii) that:  

\begin{equation}
\label{factII}
\mbox{If}\;\; a_j< a_i\;\; {\rm then }\;\; a_i \to a_j=a_j.
\end{equation}

Therefore $m_{(i,j)}= a_j$ for $j< i$, and since $0< a_i $ for $i\geq 1$ then $a_i \to 0=0$, so $m_{(i,0)}=0$ for $i \geq 1$.

\begin{lemma} {\rm \cite{Abad2010variety} (Lemma 2.4)}  \label{0 implica 1 = 0}
	Let $\LL$ be a semi-Heyting chain. The following conditions are equivalent:
	\begin{enumerate}  [\rm(a)]
		\item $0 \to a = 0$ for some $a \in L$ with $a \not= 0$.
		\item $0 \to b = 0$ for every $b \in L$ with $b \not= 0$.
	\end{enumerate}
\end{lemma}

As a particular case, we have:
\begin{equation} \label{factIII}
\mbox{If\ } 0\to 1=0 \mbox{\ then\ } 0\to a_i=0 \mbox{ \ for every\ } a_i, \mbox{\ with\ } i>0,
\end{equation}

so if $m_{(0,n-1)}=0$  then $m_{(0,i)}= 0$ for $1\leq i\leq n-1$.

\begin{lemma} {\rm  \cite{Abad2010variety}(Lemma 2.5)}  \label{implicacion}
	Let $\LL$ be a Semi-Heyting chain and let $a,b,c \in L$, $a \not= 1$. If $a \to 1 = b$ then for $c > a$
	
	\hspace{1cm}$\left\{
	\begin{array}{ll}
	a \to c = b & \hbox{if \ $b<c$;} \\
	a \to c \in [c) & \hbox{if \ $b \ge c$.}
	\end{array}
	\right. $ \end{lemma}

As a particular case: 
\begin{multline}
\label{factIV}
\mbox{If\ } 0\to 1= a_i \mbox{\ then for\ } a_j>0:\\ 0\to a_j=a_i \mbox{\ for\ } i< j \mbox{ \ and\ } 0\to a_j\in [a_j) \mbox{\ for\ } j\leq i.
\end{multline}
Therefore, if $m_{(0,n-1)}= a_i$ then $m_{(0,j)}=a_i$ for $i<j$ and $m_{(0,j)}\in [a_j)$ for $j\leq i$.

\medskip
 	 
For $n\ge 2$, let  ${\cal S}(n-1)$ be the set of matrices of size $(n-1) \times (n-1)$ such that they define an algebra in $\SH$ over the chain $a_1<a_2<\cdots < a_{n-1}=1$ and let ${\cal M}(n)$ be the set of matrices of size  $n \times n$ defining a semi-Heyting algebra over  $C_n$. 

\begin{lemma}\label{submatrix}
	If  ${\bf C}_n=\langle C_n,\land,\lor,\to, 0,1\rangle\in \SH$, then $${\bf S}_{n-1}=\langle C_n\setminus\{a_0\},\land,\lor,\to,a_1, 1\rangle\in \SH.$$
\end{lemma}

\begin{proof}
It is clear that $\langle C_n\setminus\{a_0\},\land,\lor\rangle$ is a sublattice of ${\bf C}_n$, with top element $1=a_{n-1}$ and bottom element $a_1$. We need to prove that for all $a,b\in C_n\setminus\{a_0\}=[a_1)$, then $a\to b\neq a_0$. If $ a \to b = a_0=0$ then by (SH2), $ 0= a \land 0= a \land (a\to b)= a \land b$, contradicting the fact that $a$ and $b$ are not $0$.

Since the identities (SH2), (SH3), and (SH4) hold in ${\bf C}_n$, they hold in ${\bf S}_{n-1}$ as well.
\end{proof}

It follows from Lemma \ref{submatrix} that if
 $M \in  {\cal M}(n)$ then the submatrix $S=(s_{(i,j)})$ of $M$ defined by $s_{(i,j)}=m_{(i,j)}$, for $1\leq i,j\leq n-1$ belongs to 
 ${\cal S}(n-1)$.

$$
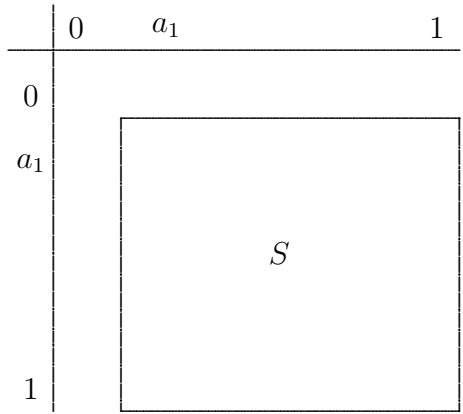
\begin{minipage}{10cm} 
\beginpicture
\setcoordinatesystem units <3mm,3mm>
\setplotarea x from 14 to 32, y from 17 to 30
\put {$ 0$} [c] at  17 33 
\put {$ a_1$} [c] at  21 33 
\put {$1$} [c] at  33 33 
\put {$ 0$} [c] at  15 30
 \put {$ a_1$} [c] at  15 27
\put {$ 1$} [c] at  15 17

\put {$ S$} [c] at  26 23

\setlinear \plot 16 16  16  34 /
\setlinear \plot 14 32 34 32  /
\setlinear \plot 19 16  19  29  34   29  34 16 19 16  /
\endpicture
\end{minipage}
$$

\section{Construction and counting}

In this section we will give a recursive construction to obtain all the semi-Heyting algebras definable on the chain with $n$ elements, or equivalently, all the matrices in ${\cal M}(n)$. This will provide with a simple way of counting them.

It is clear that renaming the elements in a finite chain, there is a bijection between the sets ${\cal M}(n)$ and ${\cal S}(n)$ for every natural number $n\ge 2$, so they have the same number of elements.

 Given  $S=(s_{(i,j)})\in {\cal S}(n-1)$, we denominate $\to_S$ the implication operation on the chain $a_1<a_2<\cdots < a_{n-1}=1$ defined by the table given by $S$. Starting from $S$ we will define matrices  $M=M(S)$ of size $n\times n$, so that they verify the conditions: 
 \begin{equation} \label{factV}
a_i\to_M a_j = a_i\to_S a_j, \; 1\leq i,j \leq n-1,
 \end{equation}
 \begin{equation} \label{factVI}
 0\to_M 0=1 \mbox{\ and}
 \end{equation}
 \begin{equation} \label{factVII}
a_i\to_M 0=0 \mbox{\ for\ } 1\leq i \leq n-1.
 \end{equation}

We may depict the part of the new matrix we have defined so far: 
$$
\begin{minipage} {10cm}
\beginpicture
\setcoordinatesystem units <3mm,3mm>
\setplotarea x from 14 to 32, y from 17 to 30
\put {$ 0$} [c] at  17 33 
\put {$ a_1$} [c] at  21 33 
\put {$1$} [c] at  33 33
\put {$\ldots$} [c] at  26 33
\put {$ 0$} [c] at  15 30
 \put {$ a_1$} [c] at  15 28
\put {$ 1$} [c] at  15 17

\put {$ 1$} [c] at  17 30 
\put {$ \vdots$} [c] at  15 23 
\put {$ 1$} [c] at  17 30  

\put {$ 0$} [c] at  17 28
\put {$ 0$} [c] at  17 17 
\put {$0$} [c] at  17 20 
 \put {$\vdots$} [c] at  17 23 
\put {$0$} [c] at  17 26 

\put {$ S$} [c] at  26 23

\setlinear \plot 16 16  16  34 /
\setlinear \plot 14 32 34 32  /
\setlinear \plot 19 16  19  29  34   29  34 16 19 16  /
\endpicture
\end{minipage}
$$

 From conditions \eqref{factV} and \eqref{factVI} it follows that  (SH4) holds. Now we want to check that the identity (SH2): $x\wedge (x \to_M y) \approx x\wedge y$ holds as well. If $x=0$ or \linebreak[3] $x,y\in \{a_1,\cdots a_{n-1}=1\}$, (SH2) holds and if $x=a_i$ with $i>0$, $a_i\wedge (a_i\to_M 0){\stackrel{{\eqref{factVII}}}{=}}a_i\wedge 0.$

To check that (SH3),  $x\wedge (y \to_M z) \approx x\wedge[(x\wedge y)\to_M (x\wedge z)]$ holds, we consider the following cases:
\begin{itemize}
	\item If $x=0$ or $x,y,z \neq 0$ then clearly the equation holds.
	\item  If $x=a_i,$ $i>0$, $y=z=0$, then
	$x\wedge (0 \to_M 0) {\stackrel{{\eqref{factV}}}{=}}x\wedge 1=x$ and $x\wedge[(x\wedge 0)\to_M (x\wedge 0)]= x\wedge (0\to_M 0){\stackrel{\eqref{factVI}}{=}}x\wedge 1=x$.
	\item If $x>0$ and $y>0$, and $z=0$ then
	$x\wedge (y \to_M 0) {\stackrel{{\eqref{factVII}}}{=}}x\wedge 0=0$ and
	$x\wedge[(x\wedge y)\to_M (x\wedge 0)]=x\wedge[(x\wedge y)\to_M 0]{\stackrel{{\eqref{factVII}}}{=}}x\wedge 0=0$.
\end{itemize}

  For the remaining case, when $x>0, y=0$ and $z>0$  we must prove that $x\land (0\to_M z)=x\land[(x\land 0)\to_M(x\land z)]$, this is, 
  $x\land (0\to_M z)=x\land[0\to_M(x\land z)]$. Here we need to consider how the first row of the matrix is defined. Using as a guide for the definition the Lemma \ref{implicacion}, and more precisely, its consequence indicated in \eqref{factIV}, we define the operation $\to_M$ and prove that (SH3) holds. This proves that all the implication operations thus defined yield semi-Heyting algebras. This, together with Lemma \ref{implicacion}, proves that all the finite semi-Heyting chains are obtained this way.

 For a fixed $j$ with $0\le j\le n-1$, D1) $0\to_M a_h=a_j$, for every  $h>j$, D2) For each $h$ such that $0<h\le j$, $0\to_M a_h=a_{i(h)}$ for some  $i(h)$ such that $h\le i(h)$. this is:
$$
\begin{minipage}{10cm} 
\beginpicture
\setcoordinatesystem units <3mm,3mm>
\setplotarea x from 14 to 32, y from 17 to 30

\put {$ 0$} [c] at  17 33 
\put {$ a_1$} [c] at  21 33 
\put {$ a_j$} [c] at  25 33 
\put {$1$} [c] at  33 33

\put {$\scriptstyle\ge a_1$} [c] at  21 30 
\put {$\scriptstyle\ge a_j$} [c] at  25 30

\put {$ 0$} [c] at  15 30
 \put {$ a_1$} [c] at  15 28
\put {$ 1$} [c] at  15 17

\put {$ 1$} [c] at  17 30 
 
\put {$ a_j$} [c] at  33 30
\put {$ a_j$} [c] at  27 30
\put {$ a_j$} [c] at  29 30
\put {$ \cdots$} [c] at  31 30

\put {$ S$} [c] at  26 23  
\put {$ 0$} [c] at  17 28
\put {$ 0$} [c] at  17 17 
\put {$0$} [c] at  17 20 
 \put {$0$} [c] at  17 23 
\put {$0$} [c] at  17 26

\setlinear \plot 16 16  16  34 /
\setlinear \plot 14 32 34 32  /
\setlinear \plot 19 16  19  29  34   29  34 16 19 16  /
\endpicture
\end{minipage}
$$

Let $x=a_i>0, y=0$ and $z=a_k>0$. 
\begin{itemize}
	\item If $i\le k$, we have:
	\begin{itemize}
		\item $k\le j$. Then $a_i\le a_k\le 0\to a_k$. So $a_i\land(0\to_M a_k)=a_i$ and $a_i\wedge[ 0\to_M (a_i\wedge a_k)]=a_i\wedge[ 0\to_M a_i]=a_i$, since $a_i\le 0\to_M a_i$.
		\item $i\le j<k$. Under these conditions, $a_i\land(0\to_M a_k)=a_i\land a_j=a_i$, while $a_i\wedge[ 0\to_M (a_i\wedge a_k)]=a_i\wedge[ 0\to_M a_i]=a_i$.
		\item $j<i\le k$. Now we have  $a_i\land(0\to_M a_k)=a_i\land a_j=a_j$ while $a_i\wedge[ 0\to_M (a_i\wedge a_k)]=a_i\wedge[ 0\to_M a_i]=a_i\land a_j=a_j$.
	\end{itemize}
			 \item If $k<i$, then $a_i\wedge[ 0\to_M (a_i\wedge a_k)]=a_i\wedge( 0\to_M a_k)$.
\end{itemize}

%

\

We now count how many of different possible definitions there are. For each fixed $j$, $0\le j\le n-1$, $n-(j+1)$ places in the first row are filled with the value $a_j$, the first place is filled by $1$, and there remain $j$ places corresponding to the values of $0\to_Ma_i$ with $1\le i\le j$. Each of these last places can be filled with any element in $[a_i)$, so there are $n-i$ possibilities. We have then, for each $j$, $(n-1)(n-2)\cdots (n-j)= \frac{(n-1)!}{(n-(j+1))!}$ different possible first rows for the matrix $M$. Adding them all, we get a total of 
$\sum\limits_{j=0}^{n-1}\frac{(n-1)!}{(n-(j+1))!}$. Replacing the variable $j$ with $i=n-(j+1)$, we can write this number as $\sum\limits_{i=0}^{n-1}\frac{(n-1)!}{i!}$

Since $N(n) =|\mathcal{M}(n)|$, it is clear that  $N(1)=1$ and that if we add the same first row to two different matrices in $\mathcal{S}(n-1)$, we get two different matrices in $\mathcal {M}(n)$.

Thus we have proved the following Lemma: 
\begin{lemma}\label{form} If  $n \ge 2$ then  
\begin{equation}\label{eqL}
 N(n) =\left(\sum\limits_{i=0}^{n-1} \frac{(n-1)!}{i!}\right) N(n-1).
\end{equation}
\end{lemma}

\begin{corollary} If $n\geq 2$, $N(n)$ is even.
\end{corollary}

 By Lemma \ref{form} we have:
$$
\begin{minipage}{10cm} 
\begin{tabular}{c||l|r}
$n$&&$N(n)$\\\hline\hline
$1$&&$1$\\\hline
$2$&$2\times 1$&$2$\\\hline
$3$&$5 \times 2$&$10$\\\hline
$4$&$16 \times 10$&$ 160$\\\hline
$5$&$65 \times 160$&$ 10{.}400$\\\hline
$6$&$326 \times 10{.}400$&$ 3{.}390{.}400$\\\hline
$7$&$1{.}957 \times 3{.}390{.}400$&$ 6{.}635{.}012{.}800$\\\hline
\end{tabular}
\end{minipage}
$$

 In {\rm \cite{Abad2010variety}} it was proved that for $n\ge 2$:
\begin{equation} \label{eqF}
N(n)=\prod\limits_{i=0}^{n-2}\left[ 1+(n-i-1)! \sum\limits_{j=i+1}^{n-1} \frac{1}{(n-j-1)!}\right].
\end{equation}

We will now derive now the formula \eqref{eqL} from \eqref{eqF}.

Since  $N(1)=1$ and we know from  \cite{sankappanavar2008semi} that $N(2)=2$, we can observe that if $n=2$ then $\sum\limits_{k=0}^{1} \frac{1!}{k!}=2$, so $\left(\sum\limits_{k=0}^{1} \frac{1!}{k!}\right)N(1)=2\times 1=N(2)$.  

Assume now that  $n\geq 3$. Writing out the first factor in the product, we get:  
$$N(n)=\prod\limits_{i=0}^{n-2}\left[ 1+ \sum\limits_{j=i+1}^{n-1} \frac{(n-i-1)!}{(n-j-1)!}\right]=$$ 
$$\left[1 + \sum\limits_{j=1}^{n-1} \frac{(n-1)!}{(n-j-1)!}\right]\prod\limits_{i=1}^{n-2}\left[ 1+ \sum\limits_{j=i+1}^{n-1} \frac{(n-i-1)!}{(n-j-1)!}\right].$$
We give names to these two factors: $$N_0(n)=1 + \sum\limits_{j=1}^{n-1} \frac{(n-1)!}{(n-j-1)!}$$ and
$$ N_r(n)=\prod\limits_{i=1}^{n-2}\left[ 1+ \sum\limits_{j=i+1}^{n-1} \frac{(n-i-1)!}{(n-j-1)!}\right].$$  

Now we calculate:
$$N_0(n)=1 + \sum\limits_{j=1}^{n-1} \frac{(n-1)!}{(n-j-1)!}=\frac{(n-1)!}{(n-1)!} + \sum\limits_{j=1}^{n-1} \frac{(n-1)!}{(n-j-1)!}=$$
$$\sum\limits_{j=0}^{n-1} \frac{(n-1)!}{(n-j-1)!}.$$

Using the variable $k=n-j-1$, if $j=0$ then $k=n-1$ and if $j=n-1$ then $k=0$, so 
$$N_0(n)=\sum\limits_{k=0}^{n-1} \frac{(n-1)!}{k!}.$$

For the other factor,   
$$ N_r(n) =\prod\limits_{i=1}^{n-2}\left[ 1+ \sum\limits_{j=i+1}^{n-1} \frac{(n-i-1)!}{(n-j-1)!}\right],$$
we put $h=i-1$  so $i=h+1$ and if $i=1$ then $h=0$ and if $i=n-2$ then $h=n-3$. We also write $k=j-1$, so $j=k+1$ and if $j=i+1$ then  $k=i=h+1$ while if $j=n-1$ then $k=n-2$, thus
$$ N_r(n) =\prod\limits_{h=0}^{n-3}\left[ 1+ \sum\limits_{k=h+1}^{n-2} \frac{(n-(h+1)-1)!}{(n-(k+1)-1)!}\right]{\stackrel{{\eqref{eqF}}}{=}}N(n-1).$$
Therefore
$$ N(n) =\left[\sum\limits_{k=0}^{n-1} \frac{(n-1)!}{k!}\right] N(n-1).$$

\bibliographystyle{alpha}

\end{document}